\documentclass[11pt, leqno]{amsart}
\usepackage{amsmath, amssymb, latexsym}
\usepackage{mathrsfs}
\usepackage{enumerate}
\usepackage{color}
\usepackage[colorlinks=true, pdfstartview=FitV, linkcolor=blue,%
citecolor=blue, urlcolor=blue]{hyperref}
\usepackage[all, knot]{xy}
\usepackage{chngcntr}
\counterwithout{figure}{section}
\xyoption{arc}
\usepackage{tikz}
\usepackage{mathdots}
\usepackage{amsmath}
\numberwithin{equation}{section}

\newtheorem{theorem}{Theorem}[section]
\newtheorem{corollary}[theorem]{Corollary}
\newtheorem{lemma}[theorem]{Lemma}
\newtheorem{proposition}[theorem]{Proposition}

\newtheorem{example}[theorem]{Example}

\theoremstyle{definition}
\newtheorem{definition}[theorem]{Definition}

\newcommand{\mrm}{\mathrm}

\newcommand{\mbf}{\mathbf}
\newcommand{\mbb}{\mathbb}
\newcommand{\mcal}{\mathcal}

\title[ Affine flag varifeties  of type $D$]
{Affine Flag varieties  of type $D$}

\author[Quanyong Chen]{Quanyong Chen}
\address{Harbin Engineering University,
Harbin, China}
\email{chenquanyong@hrbeu.edu.cn}

\author[Zhaobing Fan]{Zhaobing Fan}
\address{Harbin Engineering University,
Harbin, China}
\email{fanzhaobing@hrbeu.edu.cn}

\author[Qi Wang]{Qi Wang*}
\address{Harbin Engineering University,
	Harbin, China}
\email{wang777@ustc.edu.cn}
\thanks{ }

\begin{document}
\begin{abstract}
  The Hecke algebras and quantum group of affine type $A$ admit geometric realizations in terms of complete flags and partial flags over a local field, respectively.
  Subsequently, it is demonstrated that the quantum group associated to partial flag varieties of affine type $C$ is a coideal subalgebra of
  quantum group of affine type $A$.
  In this paper, we establish a lattice presentation of the complete (partial) flag varieties of affine type $D$.
  Additionally, we determine the structures of convolution algebra associated to complete flag varieties of affine type $D$, which is isomorphic to the (extended) affine Hecke algebra.
  We also show that there exists a monomial basis and a canonical basis of the convolution algebra, and establish the positivity properties of the canonical basis with respect to multiplication.
\end{abstract}

\maketitle

\setcounter{tocdepth}{2}
\tableofcontents
\section{Introduction}

The geometric realization of Hecke algebra has played important roles in geometric representation theory.
Iwahori \cite{Iw64} provided a geometric realization of Hecke algebras as convolution algebras $\mathbf{H}_{A}$ on pairs of complete flags over a finite field.
Soon after, Iwahori and Matsumoto \cite{IM65} realized the affine Hecke algebras $\mathbf{H}_{\widetilde{A}}$ by  utilizing pairs of complete flags of affine type over a local field.
These works are the foundation of geometric representation theory.

The geometric realization of quantum groups and Hecke algebras has always been a topic of great interest and significance.
Beilinson, Lusztig and McPherson \cite{BLM90} made a significant contribution by constructing a geometric realization of quantum Schur algebra $\mathbf{S}_{n,d}^{A}$ as convolution algebras on pairs of partial flags over a finite field.
  They also realized the (modified) quantum group $\mathbf{U}(\mathfrak{gl}_n)$ in the process of the stabilization and completion of quantum Schur algebras, and showed the
 modified quantum group $\dot{\mathbf{U}}(\mathfrak{gl}_n)$ admits a canonical basis.
In a subsequent work by Grojnowski and Lusztig \cite{GL92}, the Schur-Jimbo duality is realized
geometrically by considering the product variety of the complete flag varieties and the $n$-step partial flag varieties of type $A$.
The affine quantum Schur algebra $\mathbf{S}_{n,d}$ is by definition the convolution algebra of pairs of flags of affine type $A$ in \cite{Lu99,Lu00}.
Also, there is an affine version of Schur-Jimbo duality formed in \cite{CP96}.
Motivated by \cite{BW13}, Bao, Kujawa, Li and Wang \cite{BKLW14, BLW14} provided a geometric construction of Schur-type algebras $\mathbf{S}_{n,d}^{\imath}$ and Hecke algebras $\mathbf{H}_{C}$ in terms of $n$-step partial flags and complete flags of type $C_d$, respectively.
Fan and Li \cite{FL14} established a new duality between the Quantum algebra $\mathcal{S}^m$ and the Iwahori-Hecke algebra $\mathbf{H}_D$ of type $D$ attached to ${\rm SO}_F(2d)$ algebraically and geometrically by considering the (partial) flag varieties of type $D$.

Recall that there is a lattice representation of the complete and $n$-step flag varieties of affine type $C$ over a local field in \cite{Sa99}.
The affine Schur algebra $\mathbf{S}_{n,d}^{\mathfrak{c}}$ (resp. affine Hecke algebra $\mathbf{H}_{\widetilde{C}}$) are by the definition the convolution algebra of pairs of partial (resp. complete) flags of affine type $C$ \cite{FLLLW20}.
The crucial point to study the structure of the $\mathbf{S}_{n,d}^{\mathfrak{c}}$, as long as the corresponding i-quantum group, for example, generators, relations and the canonical basis etc., is the lattice presentation of the partial flag varieties.
The quantum algebra $\mathbf{U}_n^{\mathfrak{c}}$ is by definition a suitable subalgebras of the projective limit of the projective system of Lusztig algebras, and the comultiplication homomorphisms gives rise to show that $\mathbf{U}_n^{\mathfrak{c}}$ is a coideal subalgebra of $\mathbf{U}(\widehat{\mathfrak{sl}}_n)$.

To this end, it is compelling to ask what happens to the classical case of the  affine type $D$. The purpose of this paper is to provide an answer to this question, as a sequel to \cite{FL14, FLLLW20}.
In section 2, we recall some results of flag varieties of type $D$ over a finite field.
In section 3, our first main result is the construction of lattice presentation which can be adapted to affine type $D$, on which the special orthogonal group ${\rm SO}_F(V)$ (where $F=\mbb F((\varepsilon))$) acts for the complete flag varieties $\mcal Y_d^{\mathfrak{d}}$ and for the $n$-step partial flag varieties $\mathcal{X}_{n,d}^{\mathfrak{d}}$, which is formulated in this paper, for $n$ even. This lattice presentation can be used to study the affine Schur algebra and the corresponding i-quantum group, which has been study in forthcoming paper \cite{CF}.
In section 4, we parameterize the orbits for the product $\mathcal{Y}_{d}^{\mathfrak{d}}\times \mathcal{Y}_{d}^{\mathfrak{d}}$ under the diagonal action of the group ${\rm SO}_F(V)$ by the set of matrices.
We show that the affine Hecke algebra $\mathcal{H}_{d}^{\mathfrak{d}}$ is by the definition the convolution algebra of pairs of the complete flags in $\mathcal{Y}_{d}^{\mathfrak{d}}$, and admits a monomial basis and a canonical basis, which enjoys a positivity with respect to multiplication.

In this paper, we denote by $\mbb N$ and $[a,b]$ the set of nonnegative integers and the set of integers between $a$ and $b$, respectively.

\noindent {\bf Acknowledgement.}
We express our gratitude Haitao Ma for valuable discussions and comments.
This paper was partially supported by the NSF of China grant
12271120 and 12101152, the NSF of Heilongjiang Province grant JQ2020A001, and the Fundamental Research Funds for the central universities.

\section{Flag varieties of type $D$}
In this short section, we introduce general conventions, fix some notation, and offer a brief review of some definitions and facts for flag varieties of type $D$ over a finite field. For more details, we refer the reader to \cite{W97}.

Let $\mbb F$ be a finite field of $q$ elements with odd characteristic.
Fix a positive integer $d$ and set
$D=2d$.
Moreover, we fix a symmetric bilinear form $\overline{Q}$ on $\mathbb{F}^D$ whose associated matrix under the standard basis is
\begin{equation}\label{associated matrix}
J=\begin{bmatrix}
0 & 0 & \cdots & 0 & 1 \\
0 & 0 & \cdots & 1 & 0 \\
\vdots & \vdots & \iddots & \vdots & \vdots \\
0 & 1 & \cdots & 0 & 0 \\
1 & 0 & \cdots & 0 & 0 \\
\end{bmatrix}.
\end{equation}
For a vector subspace $W$ of $\mathbb{F}^D$, we write $|W|$ and $W^\bot$ for its dimension and orthogonal complement, respectively. A vector subspace $W$ is called isotropic if $W \subset W^\bot$.
 For any isotropic subspace $W$,
 the bilinear form $\overline{Q}$ induces a non-degenerate symmetric bilinear form $\overline{Q}|_{W^\bot/W}$ on $W^{\bot}/ W$.
 Moreover,
the associated matrix of $\overline{Q}|_{W^\bot/W}$ is of the form \eqref{associated matrix} with rank $D-2|W|$ under a certain basis.

 Denote by ${\rm O}_{\mbb{F}}(D)$ and ${\rm SO}_{\mbb{F}}(D)$ the orthogonal group and the special orthogonal group with respect to $\overline{Q}$, respectively. We have the following propositions.
\begin{proposition}
Let $W$ and $W'$ be isotropic subspaces with dimension $d-1$ and $f: W\rightarrow W'$ an invertible transformation.
  Then there exists $g \in {\rm SO}_{\mbb{F}}(D)$
such that $g|_{W}=f$.
\end{proposition}

\begin{proposition}\label{different}
Let $W$ be an isotropic subspace with dimension $d-1$. Then there exist exactly two maximal isotropic subspaces $V_1$ and $V_2$ containing $W$. Moreover, these two maximal isotropic subspaces are in different ${\rm SO}_{\mbb F}(D)$-orbits.
\end{proposition}

\begin{proposition}\label{if and only if}
 Let $W$ and $ W'$ be two maximal isotropic subspaces. Then
 $$|W/W\cap W'| \equiv 0 \ {\rm mod} \ 2$$
 if and only if there exists $g \in {\rm SO}_{\mbb F}(D)$ such that $g W=W'$.
\end{proposition}

Fix a maximal isotropic subspace $M$.
Let $\mathscr{Y}$ be the set of filtrations as following:
\begin{align*}
\mathscr{Y}=\left\{F=(F_i)_{0\leq i\leq D} \ \big| \ |F_i|=i,   \ F_i=F_{D-i}^\bot \ {\rm and} \ |F_d\cap M|\equiv 0 \ {\rm mod} \ 2 \right\}.
\end{align*}
 By Proposition \ref{if and only if}, ${\rm SO}_{\mbb F}(D)$ acts on $\mathscr{Y}$ component wisely, i.e., $(g F)_{i}=g \cdot F_i$.
Moreover, ${\rm SO}_{\mbb F}(D)$ acts transitively on $\mathscr{Y}$
thanks to the condition $|F_d\cap M|\equiv 0 \ {\rm mod} \ 2$.

We consider the stable subgroup $\mathrm{B}$ of $F$ for some $F=(F_i)_{0\leq i\leq D} \in \mathscr{Y}$.
Suppose
$\{v_1,\cdots,v_D\}$ is a basis of ${\mbb F}^D$ such that $\overline{Q}(v_i,v_j)=\delta_{i,D+1-j}$ and $\{v_1,\cdots ,v_i\}$
a basis of $F_i$, for $1 \leq i \leq D$.
Denote by $\mathrm {B}_i$ the subgroup of ${\rm SO}_{\mbb F}(D)$ such that
 $$\mrm{B}_i=\{g \in {\rm SO}_{\mbb F}(D) \mid g F_i=F_i\}.$$
 It is easy to see that $\mathrm {B}=\cap_{1\leq i\leq D}\mrm{B}_i$ is a Borel subgroup and $\mathscr{Y}\simeq {\rm SO}_{\mbb F}(D)/ \mathrm {B}$.

\section{Lattice presentation of affine flag varieties of type $D$}
In this section, we will establish a lattice presentation of the complete (partial) flag varieties of affine type $D$.

Let $F=\mbb F((\varepsilon))$ be the field of formal Laurent series over $\mbb F$ and
$\mathfrak{o}=\mbb F[[\varepsilon]]$ the ring of formal power series. Denote by $\mathfrak{m}$ the maximal ideal of $\mathfrak{o}$ generated by
$\varepsilon$.

Let $V=F^D$ be a vector space with a symmetric bilinear form $Q: V\times V \rightarrow F$ whose associated matrix under the standard basis is of the form \eqref{associated matrix}.
A free $\mathfrak{o}$-submodule $\mathcal{L}$ of $V$ with rank $D$ is called an $\mathfrak{o}$-lattice. Clearly, an $\mathfrak{o}$-basis of $\mcal L$ is also an $F$-basis of $V$.
For any lattice $\mcal L$ of $V$, we set
\begin{equation*}
  \mathcal{L}^\sharp =\{v\in V \mid Q(v,\mathcal{L}) \subset\mathfrak{o}\}, \ \
  \mathcal{L}^\ast =\{v\in V \mid Q(v,\mathcal{L}) \subset\mathfrak{m}\}.
\end{equation*}
The $\mathfrak{o}$-modules $\mathcal{L}^\sharp, \mathcal{L}^\ast$ are also lattices of $V$.
It is straightforward to show that $(\mathcal{L}^\sharp)^\sharp=\mathcal{L}$ and $\mathcal{L}^\ast=\varepsilon\mathcal{L}^\sharp$.
 For any two lattices $\mathcal{L}$ and $\mathcal{M}$, the following equations hold:
\begin{equation}\label{eqsharp}
\begin{split}
  (\mathcal{L}+\mathcal{M})^\sharp &=\mathcal{L}^\sharp \cap \mathcal{M}^\sharp, \ \   (\mathcal{L}\cap
\mathcal{M})^\sharp=\mathcal{L}^\sharp+\mathcal{M}^\sharp; \\
(\mathcal{L}+\mathcal{M})^\ast &=\mathcal{L}^\ast \cap \mathcal{M}^\ast, \ \   (\mathcal{L}\cap
\mathcal{M})^\ast=\mathcal{L}^\ast+\mathcal{M}^\ast.
\end{split}
\end{equation}
Moreover, for any $g \in {\rm O}_{F}(V)$, we have
\begin{equation}\label{orthogonal equivalent}
g \mathcal{L}^\sharp= (g \mathcal{L})^\sharp , \ \   g\mathcal{L}^\ast =(g \mathcal{L})^\ast.
\end{equation}
\subsection{Primitive lattice}

A lattice $\mathcal{L}$ is called $\mathfrak{o}$-valued with respect to $Q$ if $Q(\mathcal{L},\mathcal{L}) \subset \mathfrak{o}$.
 Suppose $\mcal L$ is an $\mathfrak{o}$-valued, we can define an induced symmetric $\mbb F$-bilinear form
$$\overline{Q}: \mathcal{L}/\varepsilon\mathcal{L}\times \mathcal{L}/\varepsilon\mathcal{L}
\longrightarrow \mathfrak{o}/\mathfrak{m}\simeq \mbb F$$
 by
$\overline{Q}(x+\varepsilon\mathcal{L},y+\varepsilon\mathcal{L})=Q(x,y)|_{\varepsilon=0}$.

\begin{definition}
 A lattice $\mathcal{L}$ is called primitive if $\mathcal{L}$ is $\mathfrak{o}$-valued
 and  the induced symmetric bilinear form $\overline{Q}$ is non-degenerate on $\mathcal{L}/\varepsilon\mathcal{L}$.
\end{definition}

\begin{proposition}\label{dual property}
A lattice $\mcal L$ is primitive if and only if $\mathcal L=\mathcal L^{\sharp}$.
\end{proposition}
\begin{proof}
 Let $\{x_1,\cdots,x_D\}$ be a basis of $\mcal L$ and $B=(b_{ij})_{D\times D}$ the associated matrix of the bilinear form under this basis.
 Denote by $B^\ast=(b_{ij}^\ast)_{D\times D}$ the adjoint matrix of $B$.

Suppose that $\mcal L$ is primitive. Then we have $b_{ij}, b_{ij}^\ast \in \mathfrak{o}$ and ${\rm det}(B) \in \mathfrak{o}\backslash \mathfrak{m}$. We only need to verify that $\mcal L^\sharp\subset \mcal{L}$. Let $v=a_1x_1+\cdots+a_Dx_D\in  \mcal L^\sharp$.
Then $Q(x_i,v)\in \mathfrak{o}$ for $1\leq i\leq D$, which means that
\begin{equation}\label{left multiply}
B\begin{bmatrix}
a_1 \\
\vdots \\
a_D\\
\end{bmatrix} \in \mathfrak{o}^D.
\end{equation}
By left multiplying $B^\ast$ on the two sides of \eqref{left multiply}, we have
\begin{align*}
{\rm det}(B)a_i \in \mathfrak{o} \Rightarrow a_i \in \mathfrak{o} \Rightarrow \mcal L=\mcal L^\sharp.
\end{align*}

Conversely, suppose that $\mcal L=\mcal L^\sharp$. It is apparent that $\mcal L$ is $\mathfrak{o}$-valued. If the induced bilinear form $\overline{Q}$ is degenerate on $\mcal L /\varepsilon \mcal L$, then there exists a vector $v=\sum a_ix_i \in \mcal L \backslash \varepsilon \mcal L$ such that $Q(v,x_i) \in \mathfrak{m}$ for $1\leq i\leq D$. This implies that $\varepsilon^{-1}v \in \mcal L^\sharp$, which is in contradiction to $\mcal L=\mcal L^\sharp$. We complete the proof.
\end{proof}

By a similar argument as that of \cite[Theorem 1.26]{W97} and \cite[Lemma 3.1.1]{FL14}, we have the following proposition.

\begin{proposition}\label{perfect basis}
A lattice $\mcal L$ is primitive if and only if there exists a basis
$\{v_1,\cdots, v_D\}$ of $\mcal L$ such that $Q(v_i,v_j)=\delta_{i,D+1-j}$.
\end{proposition}
\begin{proof}
One side is clear.
Suppose that $\mcal L$ is a primitive lattice.
Let $\{x_1,\cdots,x_D\}$ be a basis of $\mcal L$ and $B=(b_{ij})_{D\times D}$ the associated matrix of the bilinear form under this basis.
Then we have $b_{ij} \in \mathfrak{o}$ and ${\rm det}(B) \in \mathfrak{o}\backslash \mathfrak{m}$. Without loss of generality, we may assume that $b_{11} \in \mathfrak{o} \backslash \mathfrak{m}$.
Let
$$x_1'=x_1, \ \ x_i'=x_i-b_{1,i}b_{11}^{-1}x_i \ \ {\rm for} \ 2\leq i\leq D.$$
Then $\{x_1',\cdots,x_D'\}$ is a basis of $\mcal L$ and the matrix $B'=(b_{ij}')_{D\times D}$ of the bilinear form $Q$ under this basis is
$$B'=\begin{bmatrix}
b_{11}' & 0 & \cdots & 0  \\
0 & b_{22}' & \cdots & b_{2,D}' \\
\vdots & \vdots & \ddots & \vdots  \\
0 & b_{D,2}' & \cdots & b_{D,D}' \\
\end{bmatrix}.
$$
Performing the same produce, we can find a basis $\{x_1'',\cdots,x_D''\}$ of $\mcal L$ such that the associated matrix of the bilinear form $Q$ under this basis is a diagonal matrix, denoted by $D={\rm Diag}(d_{11},\cdots,d_{DD})$, where $d_{ii} \in \mathfrak{o} \backslash \mathfrak{m}$ for $1\leq i\leq D$.
Note that there exist $c_i \in \mathfrak{o} \backslash \mathfrak{m}$ and $c_i|_{\varepsilon=0}=1$ such that $d_{ii}=s_{ii}c_i^2$, where $s_{ii}=d_{ii}|_{\varepsilon=0}$.
 According to \cite[Theorem 1.26]{W97}, there exists a basis $\{v_1 ,\cdots, v_D\}$ of $\mcal L$ such that the associated matrix under this basis is of the form \eqref{associated matrix}. We complete the proof.
\end{proof}

Let $\mcal{M}, \mcal{L}$ be two lattices such that $\varepsilon\mcal{L} \subset \mcal{M}\subset \mcal{L}=\mcal{L}^\sharp$ and $Q(\mathcal{M},\mathcal{M}) \subset \mathfrak{m}$. Then $\mathcal{M}/\varepsilon\mathcal{L}$ is an isotropic subspace of $\mathcal{L}/\varepsilon\mathcal{L}$, whose complement space is $\mathcal{M}^\ast/\varepsilon\mathcal{L}$.
In particular, if $\mathcal{M}^\ast=\mathcal{M}$, then $\mathcal{M}/\varepsilon\mathcal{L}$ is a maximal isotropic subspace.

  By a similar process of extending basis of finite type, we have the following proposition.
 \begin{proposition} \label{affine extend}
 Let $\mcal M, \mcal L$ be two lattices such that $\varepsilon \mcal L \subset \mcal M \subset \mcal L=\mcal L^\sharp$.
 Assume that $Q(\mcal M,\mcal M) \subset \mathfrak {m}$ and $ |\mcal M / \varepsilon L |= a$.
 Then there exists a
basis $\{v_1,\cdots, v_D\}$ of $\mcal L$ such that $Q(v_i,v_j)=\delta_{i,D+1-j}$ and $\{v_1,\cdots,v_a, \varepsilon v_{a+1},\cdots, \varepsilon v_D\}$ a basis of $\mcal M$.
 \end{proposition}

We now consider the set $Z$ of pair of lattices
\begin{equation*}
Z=\big\{(\mathcal{M}, \mathcal{L}) \mid \mathcal{M}=\mathcal{M}^\ast,\mathcal{L}=\mathcal{L}^\sharp, \  \varepsilon\mathcal{L}\subset \mathcal{M} \subset \mathcal{L}\big\}.
\end{equation*}
By \eqref{orthogonal equivalent}, there exists an
${\rm O}_{ F}(V)$-action on $Z$ by
$g \cdot (\mathcal{M},\mathcal{L})\mapsto (g\mathcal{M},g\mathcal{L}), \forall g \in {\rm O}_{F}(V)$. Moreover, ${\rm O}_{F}(V)$ acts transitively on $Z$ by Proposition \ref{affine extend}.

 \begin{proposition}\label{odd}
 Let $(\mcal M_0, \mcal L_0), (\mcal M, \mcal L)$ be two pairs in $Z$ such that $$|\mcal{M}/\mcal{M}\cap\mcal{M}_0|+|\mcal{L}/\mcal{L}\cap\mcal{L}_0|=1.$$
Then for $(\mcal M', \mcal L') \in Z$, we have
  \begin{equation}\label{odd formular}
  | \mcal{M}'/ \mcal M'\cap \mcal M_0|+ |\mcal{M}'/ \mcal M'\cap \mcal M|+ |\mcal{L}'/ \mcal L'\cap \mcal L_0|+| \mcal{L}'/ \mcal L'\cap \mcal L|\equiv 1 \ {\rm mod} \ 2.
  \end{equation}
 \end{proposition}
 \begin{proof}
 We shall to show that \eqref{odd formular} holds in both of the following two cases:
 \begin{equation}\label{codimension}
   ({\rm a}) \ \mcal{M}=\mcal{M}_0,  \ |\mcal{L}/\mcal{L}\cap\mcal{L}_0|=1; \ \
  ({\rm b}) \ \mcal{L}=\mcal{L}_0,  \ |\mcal{M}/\mcal{M}\cap\mcal{M}_0|=1.
 \end{equation}
  Suppose that case $({\rm a})$ in \eqref{codimension} holds. We only need to verify that
  $$|\mcal{L}'/\mcal{L}\cap\mcal{L}'|+|\mcal{L}'/\mcal{L}\cap\mcal{L}_0| \equiv 1 \ {\rm mod} \ 2.$$
 According to Proposition \ref{affine extend}, there exists a basis $\{v_1,\cdots,v_D\}$ of $\mcal{L}$ such that $Q(v_i,v_j)=\delta_{i,D+1-j}$ and $\{\varepsilon^{-1}v_1,v_2,\cdots,v_{D-1},\varepsilon v_D\}$ a basis of $\mcal{L}_0$.
  We have
 \begin{align*}
 \mcal{L}=\mcal{L}\cap\mcal{L}_0\oplus \mbb F v_D; \ \ \ \ \mcal{L}_0=\mcal{L}\cap\mcal{L}_0\oplus \mbb F\varepsilon^{-1}v_1.
 \end{align*}
 The result will be verified by showing that $\mcal{L}\cap\mcal{L}_0\cap\mathcal{L}'$ is exactly properly contained in exactly one of the $\mcal{L}\cap\mcal{L}'$ and $\mcal{L}\cap\mcal{L}_0$.
 We verified it by contradiction.
  If $\mcal{L}\cap\mcal{L}_0\cap\mcal{L}'$ is properly contained in $\mcal{L}\cap\mcal{L}'$ and $\mcal{L}'\cap\mcal{L}_0$,
 then there exist $y_1,y_2 \in \mcal{L}\cap\mcal{L}_0$ such that $v_D+y_1$ and $\varepsilon^{-1}v_1+y_2 \in \mcal{L}'$, which
 is absurd since $\mcal{L}'$ is primitive.
So we may suppose that $\mcal{L}\cap\mcal{L}_0\cap\mcal{L}'=\mcal{L}\cap\mcal{L}'=\mcal{L}_0\cap\mcal{L}'$.
It is equivalent that $\mcal{L}+\mcal{L}'=\mcal{L}_0+\mcal{L}'$, which means that  $v_D=x+y+a\varepsilon^{-1}v_1 \in \mathcal{L}$ for some $x \in \mcal{L}', y\in \mcal{L} \cap\mcal{L}_0$ and $a \in \mbb F$.
We get $a=0$ since $Q(x,x)\in \mathfrak{o}$.
Similarly, there exist $x' \in \mcal L'$ and $y' \in \mcal L\cap \mcal L_0$ such that $\varepsilon^{-1}v_1=x'+y'$.
This is a contradiction to $Q(x,x') \in \mathfrak{o}$. We complete the proof of the case $(a)$.

 The proof of case $({\rm b})$ is a counterpart in case $({\rm a})$ and so we omit it.
 \end{proof}

Similar to the argument of finite case, which is formal and not reproduced here, we have the following proposition.

 \begin{proposition}\label{different orbits}
 Let $(\mathcal{M}_0,\mathcal{L}_0),(\mathcal{M},\mathcal{L})$ be two pairs in $Z$ defined as in Proposition \ref{odd}.
 Then $(\mcal{M}_0, \mcal{L}_0), (\mathcal{M}, \mathcal{L})$ are in different ${\rm SO}_F(V)$-orbits.
 Moreover, there exists $g \in {\rm O}_F(V)\backslash {\rm SO}_F(V)$ such that $g(\mcal M_0, \mcal L_0)=(\mcal M,\mcal L)$.
 \end{proposition}

 The following proposition shows that there are exactly two ${\rm SO}_F(V)$-orbits on $Z$.

\begin{proposition}\label{iff}
 For any two pairs $(\mathcal{M}_0, \mathcal{L}_0), (\mathcal{M}, \mathcal{L})\in Z$, we have
 $$|\mathcal{M}/{\mathcal{M}\cap \mathcal{M}_0}|+|\mcal{L}/\mcal{L}\cap \mcal{L}_0|\equiv 0 \ {\rm mod} \ 2$$
 if and only if there exists $g \in {\rm SO}_{F}(V)$ such that $g (\mathcal{M}_0,\mathcal{L}_0)=(\mathcal{M},\mathcal{L})$.
\end{proposition}
\begin{proof}
 Assume that $|\mcal M/\mcal M\cap \mcal M_0|+|\mcal L/\mcal L\cap \mcal L_0| \equiv 0 \ {\rm mod} \ 2$.
   We shall define $(\mcal M_i, \mcal L_i)$ inductively
   by setting that
   $$\mcal{M}_{i+1}=\varepsilon \mcal{L}_i+\mcal{M}\cap \mcal{L}_i, \ \ \ \ \mcal{L}_{i+1}=\mcal{M}_{i+1}+\mcal{L}\cap \varepsilon^{-1}\mcal{M}_{i+1}, \ \ {\rm for} \  i \geq 0.$$
   By \eqref{eqsharp} and induction, we have $\mcal{M}_i^\ast=\mcal{M}_i, \mcal{L}_i^\sharp=\mcal{L}_i$.
 For $i, a \in \mathbb{N}$, we get
\begin{equation*}
  \mcal{M}\cap \varepsilon^{-a}\mcal{M}_i =\mcal{M}\cap \varepsilon^{-a}\mcal{L}_{i-1}
  =\mcal{M}\cap \varepsilon^{-a-1}\mcal{M}_{i-1}
  =\mcal{M}\cap \varepsilon^{-a-1}\mcal{L}_{i-2}.
\end{equation*}
This implies that
\begin{equation*}
\Big|\frac{\mcal{M}_i}{\mcal{M}_{i-1}\cap\mcal{M}_i}\Big|
=\Big|\frac{\mcal{M}\cap \mcal{L}_{i-1}}{\mcal{M}\cap \mcal{M}_{i-1}}\Big|
=\Big|\frac{\mcal{M}\cap\mcal{M}_i}{\mcal{M}\cap\mcal{M}_{i-1}}\Big|.
\end{equation*}
Similarly, we have
\begin{equation*}
  \mcal{L}\cap \varepsilon^{-a}\mcal{L}_i =
  \mcal{L}\cap\varepsilon^{-a-1}\mcal{M}_i
  =\mcal{L}\cap\varepsilon^{-a-1}\mcal{L}_{i-1}=
  \mcal{L}\cap\varepsilon^{-a-2}\mcal{M}_{i-1},
\end{equation*}
and
\begin{equation*}
  \Big|\frac{\mcal{L}_i}{\mcal{L}_i\cap\mcal{L}_{i-1}}\Big|
  =\Big|\frac{\mcal{L}\cap\varepsilon^{-1}\mcal{M}_{i-1}}{
  \mcal{L}\cap\mcal{L}_{i-1}}\Big|
   = \Big|\frac{\mcal{L}\cap\mcal{L}_i}{
   \mcal{L}\cap\mcal{L}_{i-1}}\Big|.
\end{equation*}
 Thus,
\begin{equation*}
\begin{split}
\Big|\frac{\mcal{L}}{\mcal{L}\cap\mcal{L}_0}\Big|&=\sum_{a\geq 1}\Big|\frac{\mcal{L}\cap \varepsilon^{-a}\mcal{L}_0}{\mcal{L}\cap\varepsilon^{-a+1}\mcal{L}_0}\Big|=\sum_{a\geq 1}\Big|\frac{\mcal{L}\cap \mcal{L}_a}{\mcal{L}\cap\mcal{L}_{a-1}}\Big|;\\
\Big|\frac{\mcal{M}}{\mcal{M}\cap\mcal{M}_0}\Big|&=\sum_{a\geq 1}\Big|\frac{\mcal{M}\cap \varepsilon^{-a}\mcal{M}_0}{\mcal{M}\cap\varepsilon^{-a+1}\mcal{M}_0}\Big|=\sum_{a\geq 1}\Big|\frac{\mcal{M}\cap \mcal{M}_a}{\mcal{M}\cap\mcal{M}_{a-1}}\Big|.
\end{split}
\end{equation*}
So we have the following diagram:
\begin{equation}\label{diagram}
\begin{split}
(\mcal{M}_0,\mcal{L}_0)&\longrightarrow(\mcal{M}_1,\mcal{L}_0)
\longrightarrow(\mcal{M}_1,\mcal{L}_1)\longrightarrow
(\mcal{M}_2,\mcal{L}_1)\longrightarrow(\mcal{M}_2,\mcal{L}_2)
\longrightarrow\\
\cdots&\longrightarrow(\mcal{M}_k,\mcal{L}_k)
\longrightarrow(\mcal{M}_{k+1},\mcal{L}_k)
\longrightarrow(\mcal{M}_{k+1},\mcal{L}_{k+1})
\longrightarrow\cdots.
\end{split}
\end{equation}
According to diagram \eqref{diagram} and Proposition \ref{different orbits}, there exists $g \in {\rm O}_{F}(V)\backslash {\rm SO}_{F}(V)$ such that $g(\mcal M_0, \mcal L_0)=(\mcal M, \mcal L)$.

 Conversely, suppose that there exists $g \in {\rm SO}_F(V)$ such that $g(\mcal M_0, \mcal L_0)= (\mcal M, \mcal L)$.
 Let $(\mcal M', \mcal L')$ be a pair in $Z$ such that $(\mcal M, \mcal L), (\mcal M', \mcal L')$ are two pairs defined as in Proposition \ref{different orbits}.
 If
 $$|\mcal M/\mcal M\cap \mcal M_0|+|\mcal L/\mcal L\cap \mcal L_0|\equiv 1\ {\rm mod } \ 2,$$
 then according to Proposition \ref{odd}, we have
 $$|\mcal M' /\mcal M'\cap \mcal M_0|+|\mcal L' /\mcal L'\cap \mcal L_0|\equiv 0 \ {\rm mod} \ 2.$$
 From above results, there exists $g' \in {\rm SO}_F(V)$ such that $g'(\mcal M_0,\mcal L_0)=(\mcal M', \mcal L')$.
  We have $g'g(\mcal M, \mcal L)=(\mcal M', \mcal L')$,
 which is contray to Proposition \ref{different orbits}. We complete the proof.
\end{proof}
\subsection{Complete affine flag varieties of type $D$}
Fix a pair $(\mathcal{M}, \mathcal{L}) \in Z$, we consider the set $\mathcal{Y}_d^{\mathfrak{d}}$ of all collections $\Lambda=(\Lambda_i)_{i\in \mathbb{Z}}$ of lattices in $V$ as following:
\begin{align*}
\begin{split}
\mathcal{Y}_{d}^{\mathfrak{d}}=&\Big\{ \Lambda=(\Lambda)_{i\in \mathbb{Z}} \mid
 \Lambda_i\subset \Lambda_{i+1}, \Lambda_{i}=\varepsilon \Lambda_{i+D}, \ \Lambda_{i}^\ast=\Lambda_{D-i},  \ |\Lambda_i/\Lambda_{i-1}|=1, \\ \ & {\rm for} \  i \in \mathbb{Z}, \
 |\mathcal{L}/\mathcal{L}\cap \Lambda_D|+|\mathcal{M}/\mathcal{M}\cap \Lambda_d|\equiv 0 \ {\rm mod} \ 2 \Big\}.
 \end{split}
\end{align*}
From \eqref{orthogonal equivalent} and Proposition \ref{iff},
 we define an ${\rm SO}_{F}(V)$-action on $\mathcal{Y}_{d}^{\mathfrak{d}}$ by
$g: \Lambda \mapsto g(\Lambda)=\Lambda'$, where $\Lambda'_i=g(\Lambda_i)$ for $i \in \mathbb{Z}$.

\begin{proposition}\label{GB}
The action of ${\rm SO}_{F}(V)$ on $\mathcal{Y}_{d}^{\mathfrak{d}}$ is transitive.
\end{proposition}
\begin{proof}
Let $\Lambda^{j}=(\Lambda_i^j)_{i\in \mathbb{Z}} \in \mathcal{Y}_{d}^{\mathfrak{d}}$ for $j \in \{1,2\}$.
By Proposition \ref{affine extend}, we can find bases $\{v_1^j,\cdots,v_D^j\}$ of $\Lambda_D^j$ as in Proposition \ref{perfect basis} such that $\{v_1^j,\cdots,v_i^j,\varepsilon v_{i+1}^j,\cdots,\varepsilon v_{D}^j\}$ are bases of $\Lambda_i^j$, for $ i \in [1, D]$.
 We define a linear transformation $g$ of $V$ by sending $v_i^1$ to $v_i^2$.
 It is easy to verify that $g\cdot \Lambda^1=\Lambda^2$ and $g \in {\rm O}_{F}(V)$.
  According to Proposition \ref{iff}, there exist $\rho_1, \rho_2 \in {\rm SO}_{ F}(V)$ such that $\rho_1(\mcal{M},\mcal{L})=(\Lambda_d^1, \Lambda_D^1)$ and $\rho_2(\mathcal{M},\mathcal{L})=(\Lambda_d^2,\Lambda_D^2)$.
   Thus $\rho_2\rho_1^{-1}(\Lambda_d^1,\Lambda_D^1)=(\Lambda_d^2, \Lambda_D^2)$. Moreover, we have
  $$|\Lambda_d^1/{\Lambda_d^1\cap \Lambda_d^2}|+|\Lambda_D^1/{\Lambda_D^1\cap \Lambda_D^2}|\equiv 0 \ {\rm mod} \ 2,$$
  which implies that $g \in {\rm SO}_{F}(V)$.
\end{proof}

Given $\Lambda=(\Lambda_i)_{i\in \mathbb{Z}} \in \mathcal{Y}_{d}^{\mathfrak{d}}$,  we consider the stable subgroup $\mathbf{I}_{\Lambda}$ of $\Lambda$.
Let $\{v_1,\cdots,v_D\}$ be a basis of $\Lambda_D$ defined as in Proposition \ref{perfect basis} such that $\{v_1,\cdots, v_i,\varepsilon v_{i+1},\cdots,\varepsilon v_D\}$ is a basis of $\Lambda_i$ for $1\leq i\leq D$. Denote by $\mathbf{I}_{\Lambda_i}$ the subgroups of ${\rm SO}_{ F}(V)$ such that $$\mathbf{I}_{\Lambda_i}=\{g =(g_{rs})_{D\times D}\in {\rm SO}_{ F}(V) \mid g  \Lambda_i=\Lambda_i \}.$$
For the lattice $\Lambda_i$ and $g \in \mathbf{I}_{\Lambda_i}$, we have
$$\sum g_{rs}v_s \in \Lambda_i , \ {\rm for} \ 1\leq r \leq i; \ \   \sum \varepsilon g_{rs} v_s \in \Lambda_i, \ {\rm for} \ i<r\leq D.$$
Thus, we get that
\begin{align*}
g_{rs} \in
\begin{cases}
\mathfrak{m},  &  \mbox{if} \ 1\leq r\leq i,\  i <s\leq D;\\
\mathfrak{m}^{-1}, &  \mbox{if} \ i<r \leq D,\ 1\leq s \leq i;\\
\mathfrak{o},  &  \mbox{otherwise},
\end{cases}
\end{align*}
where $\mathfrak{m}^{-1}$ is the $\mathfrak{o}$-module generated by $\varepsilon^{-1}$.
 Moreover, we have $\mbf{I}_{\Lambda}= \cap_{1\leq i\leq D}\mbf{I}_{\Lambda_i}$, and $\mathbf{I}_{\Lambda}$ is consists of the elements as the following form:
\begin{equation}
\begin{bmatrix}
\mathfrak{o} & \mathfrak{m} & \cdots & \mathfrak{m}  & \mathfrak{m} \\
\mathfrak{o} & \mathfrak{o} & \cdots & \mathfrak{m}  & \mathfrak{m} \\
\vdots & \vdots & \ddots & \vdots & \vdots \\
\mathfrak{o} & \mathfrak{o} & \cdots & \mathfrak{o} & \mathfrak{m} \\
\mathfrak{o} & \mathfrak{o} & \cdots & \mathfrak{o} & \mathfrak{o} \\
\end{bmatrix}.
\end{equation}
Hence $\mathbf{I}_{\Lambda}$ is an Iwahori subgroup and $\mathcal{Y}_{d}^{\mathfrak{d}}  \simeq {\rm SO}_{F}(V)/\mathbf{I}_{\Lambda}$.

\subsection{Partial flag varieties}
Fix an even positive integer
$n=2r$, for some $r \in \mathbb{N}$. Let $\mathcal{X}_{n,d}^{\mathfrak{d}}$ be the set of chains $L=(L_i)_{i\in \mathbb{Z}}$ of lattices in $V$ subject to the following conditions:
\begin{equation}
L_i\subset L_{i+1}, \ \ L_i=\varepsilon L_{i+n}, \ \  L_i^\ast=L_{n-i} \ \ \forall i \in \mathbb{Z}.
\end{equation}
Then ${\rm SO}_{F}(V)$ acts on $\mathcal{X}_{n,d}^{\mathfrak{d}}$ by component-wise action. Denote by $\mathcal{X}_{n,d}^{\mathfrak{d},0}$ and $\mathcal{X}_{n,d}^{\mathfrak{d},1}$ the subsets of $\mathcal{X}_{n,d}^{\mathfrak{d}}$ as following:
 \begin{align*}
  \mathcal{X}_{n,d}^{\mathfrak{d},0}&=\left\{L=(L_i)_{i\in \mathbb{Z}} \in \mathcal{X}_{n,d}^{\mathfrak{d}}\ \big| \ |\mathcal{M}/L_r\cap \mathcal{M}|+|\mathcal{L}/L_n\cap \mathcal{L}|\equiv 0 \ {\rm mod} \ 2 \right\}; \\
   \mathcal{X}_{n,d}^{\mathfrak{d},1}&=\left\{L=(L_i)_{i\in \mathbb{Z}} \in \mathcal{X}_{n,d}^{\mathfrak{d}}\ \big| \ |\mathcal{M}/L_r\cap \mathcal{M}|+|\mathcal{L}/L_n\cap \mathcal{L}| \equiv 1 \ {\rm mod } \ 2 \right\}.
 \end{align*}
 Then $\mcal{X}_{n,d}^{\mathfrak{d}}$ can be decomposed as $\mathcal{X}_{n,d}^{\mathfrak{d}}=\mathcal{X}_{n,d}^{\mathfrak{d},0}\sqcup \mathcal{X}_{n,d}^{\mathfrak{d},1}$.
 We set
\begin{align}\label{abcd}
 \Lambda_{n,d}^{\mathfrak{d}}=\left\{\lambda=(\lambda_i)_{i\in \mathbb{Z}}\in \mathbb{N}^{\mathbb{Z}} \ \Big| \ \lambda_i=\lambda_{1-i}=\lambda_{i+n}, \forall i\in \mathbb{Z}; \sum_{1\leq i\leq n}\lambda_i=D\right\}.
\end{align}
 The set $\mathcal{X}_{n,d}^{\mathfrak{d}}$ admits the following decomposition:
$$\mathcal{X}_{n,d}^\mathfrak{d}=\bigsqcup_{\mathbf{a}=(a_i)\in \Lambda_{n,d}^{\mathfrak{d}}}\mathcal{X}_{n,d}^{\mathfrak{d}}(\mathbf{a}),\ \  \text{where} \ \mathcal{X}_{n,d}^\mathfrak{d}(\mathbf{a})=\left\{L\in \mathcal{X}_{n,d}^\mathfrak{d} \ \big| \ |L_i/L_{i-1}|=a_i, \forall i \in \mathbb{Z}\right\}.$$

 For the partial flag varieties, we have the counterpart of Proposition \ref{GB}.
\begin{corollary}\label{chen}
Assume $L=(L_i)_{i\in \mathbb{Z}}, L'=(L_i')_{i\in \mathbb{Z}} \in \mathcal{X}_{n,d}^{\mathfrak{d},0} \cap \mathcal{X}_{n,d}^{\mathfrak{d}}(\mathbf{a})$ (resp. $\mathcal{X}_{n,d}^{\mathfrak{d},1}\cap \mathcal{X}_{n,d}^{\mathfrak{d}}(\mathbf{a})$)
for some $\mathbf{a}=(a_i)_{i\in \mathbb{Z}} \in  \Lambda_{n,d}^{\mathfrak{d}}$.
Then there exists $g \in {\rm SO}_{F}(V)$ such that $ g\cdot L=L'$.
\end{corollary}

 Let $L=(L_i)_{i\in \mathbb{Z}} \in \mathcal{X}_{n,d}^{\mathfrak{d}}(\mathbf{a})$ for some $\mathbf{a}=(a_i)_{i \in \mathbb{Z}} \in \Lambda_{n,d}^{\mathfrak{d}}$.
 Denote by $G_{P(\mathbf{a})}$ the stabilizer of $L$ under the basis defined as in Proposition \ref{perfect basis}.
 By the same argument as for the complete case, each element of $G_{P(\mathbf{a})}$ is of the form as following:
\begin{equation}
\begin{bmatrix}
B_1 & \mathfrak{m} & \cdots & \mathfrak{m}  & \mathfrak{m} \\
\mathfrak{o} & B_2 & \cdots & \mathfrak{m}  & \mathfrak{m} \\
\vdots & \vdots & \ddots & \vdots & \vdots\\
\mathfrak{o} & \mathfrak{o} & \cdots & B_{n-1} & \mathfrak{m} \\
\mathfrak{o} & \mathfrak{o} & \cdots & \mathfrak{o} & B_n \\
\end{bmatrix},
\end{equation}
where $B_i$ is the block of type $a_{i}\times a_{i}$ with all entries in $\mathfrak{o}$ for $ i \in [1,n]$.
 Moreover, $G_{P(\mathbf{a})}$ is a parahoric subgroup. Then we have
 \begin{equation*}
 \mathcal{X}_{n,d}^{\mathfrak{d},0} =\bigsqcup_{\mathbf{a}\in \Lambda_{n,d}^{\mathfrak{d}}}\left\{L=(L_i)_{i\in \mathbb{Z}} \in \mathcal{X}_{n,d}^{\mathfrak{d},0} \ \big| \ L \in \mathcal{X}_{n,d}(\mathbf{a}) \right\}=\bigsqcup_{\mathbf{a}\in \Lambda_{n,d}^{\mathfrak{d}}}{\rm SO}_{ F}(V)/G_{P(\mathbf{a})},
 \end{equation*}
 and $\mathcal{X}_{n,d}^{\mathfrak{d},0}\simeq \mathcal{X}_{n,d}^{\mathfrak{d},1}\simeq \bigsqcup_{\mathbf{a}\in \Lambda_{n,d}^{\mathfrak{d}}}{\rm SO}_{F}(V)/G_{P(\mathbf{a})}$.
 \section{Geometric realization of the affine Hecke algebra}
 In this section, we study the convolution algebra $\mcal{H}_{d}^{\mathfrak{d}}$ of pairs of complete flags of affine type $D$.
 We present multiplication formulas in $\mcal{H}_{d}^{\mathfrak{d}}$ with generators, and prove $\mcal{H}_d^{\mathfrak{d}}$ is isomorphic to the (extended) affine Hecke algebra of type $D$.
 \subsection{Parametrizing ${\rm SO}_{F}(V)$-orbits on $\mathcal{Y}_{d}^{\mathfrak{d}}\times \mathcal{Y}_{d}^{\mathfrak{d}}$}

 Let $\Sigma_d$ be the set of matrices with
entries being non-negative integer as following:

\begin{align*}
\Sigma_d=&\bigg\{\sigma \in  {\rm Mat}_{\mathbb{Z}\times\mathbb{Z}}(\mathbb{N}) \ \bigg| \ \sigma_{ij}=\sigma_{1-i,1-j}=
\sigma_{i+D,j+D}, \ \forall i,j \in \mathbb{Z}, \ \sum_{i}\sigma_{ij}=\sum_{j}\sigma_{ij}=1,\\
&\sum_{i\leq 0<j}\sigma_{ij}+\sum_{i\leq d<j}\sigma_{ij}\equiv 0 \ {\rm mod} \ 2 \bigg\}.
\end{align*}
Let ${\rm SO}_{F}(V)$ act diagonally on the product $\mathcal{Y}_{d}^{\mathfrak{d}} \times \mathcal{Y}_{d}^{\mathfrak{d}}$.
Thanks to the condition $$|\mathcal{L}/\mathcal{L}\cap \Lambda_D|+|\mathcal{M}/\mathcal{M}\cap \Lambda_d|\equiv 0 \ {\rm mod} \ 2,$$
 we may define a map $\Phi$ from the set of ${\rm SO}_{F}(V)$-orbits in $\mathcal{Y}_{d}^{\mathfrak{d}}\times\mathcal{Y}_{d}^{\mathfrak{d}}$ to $\Sigma_d$, by sending the orbit ${\rm SO}_{F}(V)\cdot(\Lambda,\Lambda')$ to $\sigma=(\sigma_{ij})_{i,j\in \mathbb{Z}}$, where
$$\sigma_{ij}=\Big|\frac{\Lambda_i+\Lambda_i\cap \Lambda_j'}{\Lambda_{i-1}+\Lambda_i\cap \Lambda_{j-1}'}\Big|.$$
By the definition of $\sigma_{ij}$, we have
\begin{equation*}
\begin{split}
\sigma_{1-i,1-j}&=\Big|\frac{\Lambda_{-i}+\Lambda_{1-i}\cap \Lambda_{1-j}'}{\Lambda_{-i}+\Lambda_{1-i}\cap \Lambda_{-j}'}\Big|=\Big|\frac{\Lambda_{D-i}+\Lambda_{D+1-i}\cap \Lambda_{D+1-j}'}{\Lambda_{D-i}+\Lambda_{D+1-i}\cap \Lambda_{D-j}'}\Big|\\
&=\Big|\frac{(\Lambda_{i}\cap (\Lambda_{i-1}+ \Lambda_{j-1}'))^\ast}{(\Lambda_{i}\cap (\Lambda_{i-1}+ \Lambda_{j}'))^\ast}\Big|=\Big|\frac{\Lambda_{i-1}+\Lambda_{i}\cap \Lambda_{j}'}{\Lambda_{i-1}+\Lambda_{i}\cap \Lambda_{j-1}'}\Big|=\sigma_{ij}, \ \ {\rm for} \ i,j \in \mbb Z.
\end{split}
\end{equation*}

By a similar argument as for \cite[Proposition 3.1.2]{FLLLW20} and \cite[Proposition 2.6]{H99}, we have the following proposition.

\begin{proposition}\label{fenjie}
 Let $\sigma=(\sigma_{ij})_{i,j \in \mathbb{Z}}$ be the associated matrix of $(\Lambda,\Lambda')$ under the map $\Phi$.
 Then we can decompose $V$ into $V=\oplus_{i,j \in \mathbb{Z}}V_{ij}$ as ${\mbb F}$-vector spaces satisfying that $|V_{ij} \mid =\sigma_{ij}$,
 \begin{equation*}
 \Lambda_i=\bigoplus_{k,l\in \mathbb{Z},k\leq i}V_{kl}, \ \  \Lambda_j'=\bigoplus_{k,l\in \mathbb{Z},l\leq j}V_{kl}, \ \ \  \forall i,j \in \mathbb{Z}.
 \end{equation*}
Moreover, there exists a basis $\{x_{ij}^m|1\leq m\leq a_{ij}\}$ of $V_{ij}$ such that
\begin{equation}\label{decomposition}
\begin{split}
x_{i,j}^m&=\varepsilon x_{i+D,j+D}^m, \ \forall i,j \in \mathbb{Z},\ 1\leq m\leq \sigma_{ij},\\
Q(x_{ij}^m,x_{kl}^{m'})&=Q(x_{kl}^{m'},x_{ij}^m), \ \forall i,j,k,l \in \mathbb{Z},1\leq m \leq \sigma_{ij},1\leq m' \leq \sigma_{kl},\\
Q(x_{ij}^m,x_{kl}^{m'})&=\varepsilon Q(x_{ij}^{m},x_{k+D,j+D}^{m'}), \ \forall i,j,k,l \in \mathbb{Z},1\leq m \leq \sigma_{ij},1\leq m' \leq \sigma_{kl},\\
Q(x_{ij}^m,x_{kl}^{m'})&=\delta_{m,m'}, \ \forall 1\leq i,k\leq D,i+k=D+1,j+l=D+1.
\end{split}
\end{equation}
\end{proposition}

From proposition \ref{decomposition}, we have the Iwahori-Bruhat decomposition for the group
${\rm SO}_{F}(V)$.

\begin{proposition} \label{Iso}
 The map $\Phi: {\rm SO}_{F}(V)\backslash \mathcal{Y}_{d}^{\mathfrak{d}} \times \mathcal{Y}_{d}^{\mathfrak{d}} \rightarrow \Sigma_{d}$ is a bijection.
\end{proposition}
\begin{proof}
By Proposition \ref{fenjie},  $\Phi$ is clearly surjective.
Assume that there exist two pairs $(\Lambda,\Lambda'), (\widetilde{\Lambda}, \widetilde{\Lambda}')$ of lattice chains in $\mcal{Y}_d^{\mathfrak{d}}$ such that
 $\phi(\Lambda, \Lambda')=\Phi(\widetilde{L}, \widetilde{L}')=\sigma$.
We can find bases $\{x_{ij}^m\}$ and $\{y_{ij}^m\}$ for the pairs $(\Lambda, \Lambda')$ and $(\widetilde{\Lambda}, \widetilde{\Lambda'})$, respectively, defined as in Proposition \ref{fenjie}.
We define a linear transformation $g: V\rightarrow V$ by sending $x_{ij}^m$ to $y_{ij}^m$ for $ i,j\in \mathbb{Z}, 1\leq m\leq \sigma_{ij}$.
Then we have $g \in {\rm O}_F(V)$ and $g\cdot (\Lambda,\Lambda')=(\widetilde{\Lambda},\widetilde{\Lambda'})$.
Moreover, since $\Lambda$ and $ \Lambda'$ belong to $ \mathcal{Y}_{d}^{\mathfrak{d}}$, we get $g \in {\rm SO}_{F}(V)$ from Proposition \ref{iff}. We complete the proof.
\end{proof}
\subsection{Convolution algebra}
Let $v$ be an indeterminate and $\mathcal{A}=\mathbb{Z}[v,v^{-1}]$.
We define
$$\mathcal{H}_{d;\mathcal{A}}^{\mathfrak{d}}=\mathcal{A}_{{\rm SO}_{ F}(V)}(\mathcal{Y}_{d}^{\mathfrak{d}}\times \mathcal{Y}_{d}^{\mathfrak{d}})$$
to be the space of ${\rm SO}_{F}(V)$-invariant $\mathcal{A}$-valued functions
on $\mathcal{Y}_{d}^{\mathfrak{d}}\times \mathcal{Y}_{d}^{\mathfrak{d}}$.
For $\sigma \in \Sigma_{d}$,
we denote by $[\sigma]$ the characteristic function of the corresponding orbit $\mathcal{O}_{\sigma}$.
 Then $\mathcal{H}_{d;\mathcal{A}}^{\mathfrak{d}}$ is a free $\mathcal{A}$-module with a basis $\{[\sigma] \mid \sigma \in \Sigma_{d}\}$.
We define a (generic) convolution product $\ast$ on $\mathcal{H}_{d;\mathcal{A}}^{\mathfrak{d}}$ as follows.
For a tripe of matrices $(\sigma,\sigma',\sigma'')\in \Sigma_{d}\times \Sigma_{d}\times \Sigma_{d}$,
we choose $(\Lambda,\Lambda'')\in \mathcal{O}_{\sigma''}$, and let $g_{\sigma,\sigma',\sigma'';q}$ be the number of $\Lambda' \in \mathcal{Y}_{d}^{\mathfrak{d}}$ such that $(\Lambda,\Lambda') \in \mathcal{O}_{\sigma}$ and $(\Lambda',\Lambda'') \in \mathcal{O}_{\sigma'}$.
Then there exists a polynomial $g_{\sigma,\sigma',\sigma''} \in \mathbb{Z}[v,v^{-1}]$ such that $g_{\sigma,\sigma',\sigma'';q}=g_{\sigma,\sigma',\sigma''}|_{v=\sqrt{q}}$ for every odd prime power $q$.
We define the convolution product on $\mathcal{H}_{d;\mathcal{A}}^{\mathfrak{d}}$ by letting
$$[\sigma] \ast [\sigma']=\sum_{\sigma''}g_{\sigma,\sigma',\sigma''}[\sigma''].$$
Equipped with the convolution product, the $\mathcal{A}$-module $\mathcal{H}_{d;\mathcal{A}}^{\mathfrak{d}}$ becomes an associative $\mathcal{A}$-algebra.
We set that
$$\mathcal{H}_{d}^{\mathfrak{d}}=\mathbb{Q}(v)\otimes_{\mathcal{A}}\mathcal{H}_{d;\mathcal{A}}^{\mathfrak{d}}.$$

We shall provide an explicit description of the multiplication formulas of $\mathcal{H}_{d}^{\mathfrak{d}}$.
For any $1\leq j\leq d-1$,
define the characteristic function $[T_j]$ in $\mathcal{H}_{d;\mathcal{A}}^{\mathfrak{d}}$ by
\begin{align*}
[T_j](\Lambda,\Lambda')&=\left\{\begin{array}{ll}
   1, & {\rm if}\ \Lambda_i=\Lambda'_i, \forall i \in [0,d]\backslash \{j\}, \Lambda_j\neq \Lambda'_j;\\
   0, &{\rm otherwise}.\ \\
 \end{array}
 \right. \\
[T_0](\Lambda,\Lambda')&=\left\{\begin{array}{ll}
   1, & {\rm if}\ \Lambda_i=\Lambda'_i, \forall i \in [2,d] , \Lambda_{0}\neq \Lambda'_{0}, \Lambda_{1}\neq \Lambda'_{1}, \Lambda_{-1} \subset \Lambda'_{1};\\
   0, &{\rm otherwise}.\ \\
 \end{array}
 \right. \\
[T_d](\Lambda,\Lambda')&=\left\{\begin{array}{ll}
   1, & {\rm if}\ \Lambda_i=\Lambda'_i, \forall i \in [0,d-2] , \Lambda_{d-1}\neq \Lambda'_{d-1}, \Lambda_{d}\neq \Lambda'_{d}, \Lambda_{d-1} \subset \Lambda'_{d+1};\\
   0, &{\rm otherwise}.\ \\
 \end{array}
 \right. \\
 [T_{\rho}](\Lambda,\Lambda')&=\left\{\begin{array}{ll}
   1, & {\rm if}\ \Lambda_i=\Lambda'_i, \forall i \in [1,d-1] , \Lambda_{0}\neq \Lambda'_{0}, \Lambda_{d}\neq \Lambda'_{d};\\
   0, &{\rm otherwise}.\ \\
 \end{array}
 \right.
 \end{align*}

 For $i,j \in \mathbb{Z}$, let $E^{ij}$ be the $\mathbb{Z}\times \mathbb{Z}$ matrix whose $(k,l)$-th entries are $1$, for all $(k,l)\equiv(i,j) \ {\rm mod} \ n$, and $0$ otherwise.
We set
$$E_\theta^{ij}=E^{ij}+E^{1-i,1-j}.$$
Moreover, we define a function on $\mbb Z \times \mbb Z$ by
\begin{align*}
\xi(x,y)=\left\{\begin{array}{ll}
   2, & {\rm if}\ x>y;\\
   0, &{\rm oterwise }. \\
 \end{array}
 \right.
\end{align*}

By a similar argument to \cite[Proposition 3.5]{Lu99} and \cite[Lemma 4.3.1]{FLLLW20}, we get the following lemma.

\begin{lemma}\label{multiplication}
Assume that $h \in [1,d-1]$. Let $\sigma=(\sigma_{ij})_{i,j \in \mbb Z} \in \Sigma_d$.

$(a)$ Assume that $\sigma_{h,k}=\sigma_{h+1,l}=1$. Then
\begin{equation}
[T_h]\ast [\sigma]=v^{\xi(k,l)}[\sigma-E_{\theta}^{h,k}-E_{\theta}^{h+1,l}
+E_{\theta}^{h,l}+E_{\theta}^{h+1,k}]+(v^{\xi(k,l)}-1)[\sigma].
\end{equation}

$(b)$ Assume that $\sigma_{-1,k}=\sigma_{1,l}=1$. Then
\begin{equation}
[T_0]\ast [\sigma]=v^{\xi(k,l)}[\sigma-E_{\theta}^{-1,k}-E_{\theta}^{1,l}
+E_{\theta}^{-1,l}+E_\theta^{1,l}]+(v^{\xi(k,l)}-1)[\sigma].
\end{equation}

$(c)$ Assume that $\sigma_{d-1,k}=\sigma_{d+1,l}=1$. Then
\begin{equation}
[T_d]\ast [\sigma]=v^{\xi(k,l)}[\sigma-E_{\theta}^{d-1,k}-E_{\theta}^{d+1,l}
+E_{\theta}^{d-1,l}+E_\theta^{d+1,k}]+(v^{\xi(k,l)}-1)[\sigma].
\end{equation}

$(d)$ Assume that $\sigma_{0,k}=\sigma_{d,l}=1$. Then
\begin{equation}
[T_p]\ast [\sigma]=[\sigma-E_{\theta}^{0,k}-E_{\theta}^{d,l}+E_{\theta}^{1,k}+E_\theta^{d+1,l}].
\end{equation}
\end{lemma}

A product of basis elements $[B_1]\ast \cdots \ast [B_m]$
 in $\mathcal{H}_{d;\mathcal{A}}^{\mathfrak{d}}$ is called monomial if for each $i$, $B_i$ is of the form $T_{\alpha}$ for some $\alpha \in [0,d]\cup \{p\}$.

From Lemma \ref{multiplication} and a similar argument to \cite[\S 8.6]{BF05}, we get the following proposition.
\begin{proposition}
For any $\sigma \in \Sigma_d$, there exists a monomial product $[B_1] \ast \cdots \ast [B_m] \in \mathcal{H}_{d; \mathcal{A}}^{\mathfrak{d}}$ such
that $[\sigma]=[B_1]\ast \cdots \ast [B_m]$.
\end{proposition}
\begin{proof}
  By the results of \cite[\S 8.6]{BF05}, for any matrix $\sigma=(\sigma_{ij}) \in \Sigma_d$ such that $\sum_{i\leq 0< j}\sigma_{ij}\equiv \sum_{i\leq d<j}\sigma_{ij}\equiv 0\ {\rm mod} \ 2$, there exists a monomial product such that
 \begin{equation}\label{monomial basis}
 [\sigma]=[B_1]\ast \cdots \ast [B_m],
 \end{equation}
 where $B_i$ is of the form $T_j$ for some $j \in [0,d]$.
  By Lemma \ref{multiplication}, given a matrix $\sigma \in \Sigma_d$ satisfying that $\sum_{i\leq 0< j}\sigma_{ij}\equiv \sum_{i\leq d<j}\sigma_{ij}\equiv 1\ {\rm mod}\ 2$, we have
 $$[T_p]\ast[\sigma]=[\sigma'],$$
  where $\sigma'=\sigma-E_{\theta}^{0,k}-E_{\theta}^{d,l}+E_{\theta}^{1,k}+E_\theta^{d+1,l}$.
 Then there exists a monomial product such that $[\sigma']=[B_1']\ast \cdots \ast [B_{m'}']$ is of the form \eqref{monomial basis}.
 Thus,
 $$[\sigma]=[T_p]\ast[B_1']\ast \cdots \ast [B_{m'}'].$$
 We complete the proof.
\end{proof}
\begin{example}
 Consider the matrix $\sigma=(\sigma_{ij})_{i,j \in \mbb Z} \in \Sigma_d$ as following:

\[\sigma =
\begin{tabular}{   c | c | c | c | c| c| c| c| c | c | c | c}

 & $c_{-1}$ & $c_0$ & $c_1$ & $c_2$  & $\cdots$ & $c_d$ & $c_{d+1}$ & $\cdots$ & $c_{D-1}$ & $c_D$ & $c_{D+1}$  \\
\hline
 $r_1$ & & & & & & & & & & & 1   \\
 \hline
 $r_2$ & & & & 1 & & & & & & &   \\
\hline
 $\vdots$ & & & & & $\ddots$ & & & & & &   \\
\hline
 $r_d$ & & & & & & 1 & & & & &   \\
\hline
 $r_{d+1}$ & & & & & & & 1 & & & &   \\
\hline
 $\vdots$ & & & & & & & & $\ddots$ & & &   \\
\hline
 $r_{D-1}$ & & & & & & & & & 1 & &   \\
\hline
 $r_D$ & & 1 & & & & & & & & &   \\

\end{tabular}
\]
where `$r_i$' and `$c_j$' in the table  indicate the $i$-th row and $j$-th column of the matrix $\sigma$, respectively.
We have
\begin{align*}
[\sigma]=[T_{\rho}]\ast[T_1]\ast\cdots\ast[T_{d-1}]\ast[T_d]\ast[T_{d-2}]\ast\cdots\ast[T_1].
\end{align*}
\end{example}
Recall \cite{Lu83} that the (extended) affine-Hecke algebra $\mathbf{H}_{\widetilde{D_d}}$ of type $D$ is a unital associative algebra over $\mathcal{A}$ generated by $T_i$ for $i \in [0,d]$ and $T_{\rho}$ subject to the following relations:
\begin{align}
&T_i^2=(v^2-1)T_i+v^2,  \ \ 0\leq i\leq d, \notag \\
&T_jT_{j+1}T_j=T_{j+1}T_jT_{j+1}, \ \  1\leq j<d-1, \notag\\
&T_iT_j=T_jT_i, \ \  1\leq i,j\leq d-1 \ {\rm and} \ |i-j|>1, \notag\\
&T_0T_k=T_kT_0, \ \  k\neq 2, \ \ T_dT_l=T_lT_d, \ \  l\neq d-2, \notag\\
&T_0T_2T_0=T_2T_0T_2,  \ \  T_{d-2}T_dT_{d-2}=T_dT_{d-2}T_{d}, \notag\\
&T_0=T_{\rho}T_1T_{\rho},  \ \  T_d=T_{\rho}T_{d-1}T_{\rho}, \label{formular 1}\\
&T_i=T_{\rho}T_iT_{\rho}, \ \  1<i<d-1 \label{extended formula 2} .
\end{align}

\begin{proposition}
 The assignment of sending the functions $[T_{\alpha}]$, for $\alpha \in [0,d]\cup \{p\}$,
 in the algebra $\mathcal{H}_{d;\mathcal{A}}^{\mathfrak{d}}$ to the generators $T_{\alpha}$ of $\mathbf{H}_{\widetilde{D_d}}$ in the same indexes is an isomorphism.
\end{proposition}
\begin{proof}
The relations above except the labeled ones are reduced to the finite type and hence omit it.
Since (\ref{extended formula 2}) is clear,
 we only need to prove (\ref{formular 1}) holds. Note that
 \begin{align*}
 [T_p]\ast[T_1]\ast[T_p]=\sharp G_{T_p,T_1}[T_0],
 \end{align*}
 where $G_{T_p, T_1}$ is set of the pairs $(\Lambda',\Lambda'')$ in $\mcal{Y}_d^{\mathfrak{d}} \times \mcal{Y}_d^{\mathfrak{d}}$ determined by the following conditions:
 \begin{enumerate}
 \item $\Lambda_0'=\Lambda_0''=\Lambda_1\cap \Lambda_1'$;

 \item $\Lambda_1'=\Lambda_1$ and $\Lambda_1''=\widetilde{\Lambda}_1$;

\item $\Lambda_i'=\Lambda_i''=\Lambda_i$ for $i \in [2,d]$;

\item $(\Lambda, \widetilde{\Lambda})$ is a fixed pair in $\mcal{Y}_d^{\mathfrak{d}} \times \mcal{Y}_d^{\mathfrak{d}}$ whose associated matrix is $T_0$.
\end{enumerate}
 It is clear that $\sharp G_{T_p,T_1}=1$, which implies that $[T_0]=[T_p]\ast[T_1]\ast[T_p]$.
  The left one is similar, and hence we omit it.
\end{proof}
\subsection{The canonical basis of $\mathcal{H}_{d;\mcal A}^{\mathfrak{d}}$}
Fix $\Lambda \in \mathcal{Y}_{d}^{\mathfrak{d}}$.
For $\sigma \in \Sigma_{d}$, we define
$$Y^{\Lambda}_{\sigma}=\left\{ \Lambda' \in \mathcal{Y}_{d}^{\mathfrak{d}} \mid (\Lambda,\Lambda') \in \mathcal{O}_{\sigma} \right\}.$$
This is an orbit of the stabilizer subgroup ${\rm Stab}_{{\rm SO}_{F}(V)}(\Lambda)$ of ${\rm SO}_{F}(V)$, and one can associate to it a structure of quasi-projective algebraic variety. Now, we compute its dimension $d(\sigma)$.

The following lemma analogue of \cite[Lemma 4.3]{Lu99} and \cite[Lemma 4.1.1]{FLLLW20}.
\begin{lemma}\label{matric}
 Fix $\Lambda \in \mathcal{Y}_{d}^{\mathfrak{d}}$.
 For $\sigma \in \Sigma_{d}$,
 the dimension of $Y_{\sigma}^\Lambda$ is given by
\begin{equation}\label{Dimen eq}
  d(\sigma)=\frac{1}{2}\Big(\sum_{\substack{i\geq k, j<l\\ i\in [1,D]}}\sigma_{ij}\sigma_{kl}-\sum_{i\geq 1>j}\sigma_{ij}-\sum_{i\geq d+1>j}\sigma_{ij} \Big).
\end{equation}
\end{lemma}

Define a partial order $"\leq"$ on $\Sigma_{d}$ by $\sigma \leq \sigma'$ if $\mathcal{O}_{\sigma} \subset \overline{\mathcal{O}_{\sigma'}}$.
For any $\sigma, \sigma' \in \Sigma_{d}$,
we say that $\sigma \preceq \sigma' $ if and only if
$$\sum_{k\geq i, l\leq j }\sigma_{kl} \leq \sum_{k\geq i, l\leq j}\sigma'_{kl},  \ \ \forall i>j.$$
Since the Bruhat order of affine type $D$ is compatible with the Bruhat order of affine type $A$,
we see that the partial order "$\leq$" is compatible with  the Bruhat order of affine type $D$.

Assume for now that the ground field is $\overline{\mbb F}$ of the finite field $\mbb F$.
Let $IC_{\sigma}$ be the intersection cohomology complex of the closure $\overline{Y_{\sigma}^\Lambda}$ of $Y_{\sigma}^\Lambda$,
taken in certain ambient algebraic variety over $\overline{\mbb F}$, such that the restriction of the stratum $IC_{\sigma}$ to $Y_{\sigma}^\Lambda$ is the constant sheaf on $Y_{\sigma}^\Lambda$.
 We refer to \cite{BBD82} for the precise definition of intersection complexes.
 The restriction of the $i$-th cohomology sheaf $\mathcal{H}_{Y_{\sigma}^\Lambda}^i(IC_{\sigma})$ of $IC_{\sigma}$ to $Y_{\sigma'}^\Lambda$ for $\sigma' \leq \sigma$ is a trivial local system, whose rank is denoted by $n_{\sigma',\sigma,i}$.
 Set
\begin{align*}
\{\sigma\}_d=\sum_{\sigma' \leq \sigma}P_{\sigma',\sigma}[\sigma], \  \  {\rm where} \ P_{\sigma',\sigma}=\sum_{i\in \mathbb{Z}}n_{\sigma',\sigma,i}v^{i-d(\sigma)}.
\end{align*}
The polynomials $P_{\sigma',\sigma}$ satisfy
\begin{align*}
P_{\sigma,\sigma}=1,  \ \  P_{\sigma',\sigma} \in v^{-1}\mathbb{Z}[v^{-1}] \ {\rm for \ any}\ \sigma' \leq \sigma.
\end{align*}
Recall $\left\{[\sigma] \mid \sigma \in \Sigma_{d} \right\}$ forms an $\mcal A$-basis of $\mathcal{H}_{d;\mcal A}^{\mathfrak{d}}$.
We have the following theorem.
\begin{theorem}
The set $\left\{\{\sigma\}_d\mid \sigma \in \Sigma_{d}\right\}$ forms an $\mathcal{A}$-basis of $\mathcal{H}_{d;\mathcal{A}}^{\mathfrak{d}}$, called the canonical basis. Moreover, the structure constants of $\mathcal{H}_{d;\mcal A}^{\mathfrak{d}}$ with respect to the canonical basis are in $\mathbb{N}[v,v^{-1}]$.
\end{theorem}


\begin{thebibliography}{99999}\frenchspacing
 \bibitem[BBD82]{BBD82}
 A. Beilinson, J. Bernstein, and P. Deligne, \emph{Faisceaux pervers}, Ast$\acute{\rm e}$risque 100 (1982) 5-171.
\bibitem[BF05]{BF05}
A. Bj$\ddot{\rm o}$rner, B. Francesco B, \emph{Combinatorics of Coxeter Groups}, Springer Berlin Heidelberg, 2005.
\bibitem[BKLW14]{BKLW14}
H. Bao, J. Kujawa, Y. Li, and W. Wang, \emph{Geometric Schur duality of classical type}, Transform. Groups (2014) 1-61.
\bibitem[BLM90]{BLM90}
  A. Beilinson, G. Lusztig, and R. McPherson, \emph{ A geometric setting for the quantum deformation of $GL_{n}$}, Duke Math. J. 61 (1990) 655-677.
\bibitem[BLW14]{BLW14}
H. Bao, Y. Li, and W. Wang, \emph{A geometric setting for the coideal algebra $\dot{\mathbf{U}}^{\imath}$ and compatibility of
canonical bases}, Appendix to [BKLW14], 15pp. Transform. Groups (to appear).
\bibitem[BW13]{BW13}
H. Bao, W. Wang, \emph{A new approach to Kazhdan-Lusztig theory of type $B$ via quantum symmetric pairs}, Ast$\acute{\rm e}$risque 402 (2013) vii+134pp.
\bibitem[CF]{CF}
Q. Chen, Z. Fan, \emph{Geometric approach to i-quantum
group of affine type $D$}, prepublished.
\bibitem[CP96]{CP96}
V. Chari, A. Pressley, \emph{Quantum affine algebras and affine Hecke algebras}, Pacific J. Math. 174 (1996) 295-326.
\bibitem[FL14]{FL14}
Z. Fan, Y. Li, \emph{Geometric Schur duality of classical type, II}, Trans. Amer. Math. Soc. Series B 2 (2015) 51-92.
\bibitem[FLLLW20]{FLLLW20}
Z. Fan, C. Lai, Y. Li, L. Luo, and W. Wang, \emph{Affine flag varieties and quantum symmetric pairs}, Mem. Amer. Math. Soc.
265 (2020) v+123pp.
\bibitem[GL92]{GL92}
 I. Grojnowski, G. Lusztig, \emph{On bases of irreducible representations of quantum GLn, in Kazhdan-Lusztig theory and related topics} (Chicago, IL, 1989), Contemp. Math. 139  Amer. Math.
Soc. Providence, RI, 1992, 167-174.
 \bibitem[H99]{H99}
 R. Howe, \emph{Affine-like Hecke algebras and p-adic representation theory. in Iwahori-Hecke algebras
and their representation theory} (Martina-Franca, 1999), Lecture Notes in Math., 1804, Springer, Berlin 2002, 27-69.
\bibitem[Iw64]{Iw64}
 N. Iwahori, \emph{On the structure of a Hecke ring of a Chevalley group over a finite field}, J. Fac. Sci.
Univ. Tokyo Sect. I 10 (1964) 215-236.
\bibitem[IM65]{IM65}
 N. Iwahori, H. Matsumoto, \emph{On some Bruhat decomposition and the structure of the Hecke
rings of p-adic Chevalley groups}, Inst. Hautes $\acute{{\rm E}}$tudes Sci. Publ. Math. 25 (1965) 5-48.
\bibitem[Lu83]{Lu83}
 G. Lusztig, \emph{Some examples of square integrable representations of semisimple p-adic
groups}, Trans. Amer. Math. Soc. 277  (1983) 623-653.
\bibitem[Lu99]{Lu99}
G. Lusztig, \emph{Aperiodicity in quantum affine $\mathfrak{gl}_n$}, Asian J.Math. 3 (1999) 147-177.
\bibitem[Lu00]{Lu00}
G. Lusztig, \emph{Transfer maps for quantum affine $\mathfrak{sl}_n$}, in Representations and quantizations (Shanghai 1998), China High. Educ. Press, Beijing 2000, 341-356.
\bibitem[Sa99]{Sa99}
D. Sage, \emph{The geometry of fixed point varieties of affine flag manifolds}, Trans. Amer. Math. Soc.
352 (1999) 2087-2119.
\bibitem[W97]{W97}
 Z. Wan, \emph{Geometry of classical groups over finite fields}, Science Press, 1997.
\end{thebibliography}
\end{document}